\documentclass[12pt,reqno]{amsart}
\usepackage{fullpage}
\usepackage{times}
\usepackage{amsmath,amssymb,amsthm,url}
\usepackage[utf8]{inputenc}
\usepackage[english]{babel}
\usepackage{bm}
\usepackage{graphicx}

\newtheorem{theorem}{Theorem}[section]
\newtheorem{lemma}[theorem]{Lemma}
\newtheorem{corollary}[theorem]{Corollary}
\newtheorem{proposition}[theorem]{Proposition}

\theoremstyle{definition}
\newtheorem{definition}[theorem]{Definition}

\newcommand{\B}{{\mathcal B}}

\newcommand{\N}{{\mathbb N}}

\newcommand{\R}{{\mathbb R}}

\newcommand{\Z}{{\mathbb Z}}

\newcommand{\ord}{\mathop{\rm ord}}

\newcommand{\ba}{{\bm a}}

\newcommand{\balpha}{{\bm\alpha}}

\renewcommand{\mod}[1]{{\ifmmode\text{\rm\ (mod~$#1$)}\else\discretionary{}{}{\hbox{ }}\rm(mod~$#1$)\fi}}
\newcommand{\ep}{\varepsilon}
\newsymbol\dnd 232D
\newcommand{\li}{\mathop{\rm li}}

\begin{document}

\title{Counting multiplicative groups with prescribed subgroups}
\author{Jenna Downey}
\address{Department of Mathematics and Statistics \\ C245 (Science Building) \\ 1000 KLO Road \\ Kelowna, BC, Canada \ V1Y 4X8}
\email{jdowney@okanagan.bc.ca}
\author{Greg Martin}
\address{Department of Mathematics \\ University of British Columbia \\ Room 121, 1984 Mathematics Road \\ Vancouver, BC, Canada \ V6T 1Z2}
\email{gerg@math.ubc.ca}
\maketitle

\begin{abstract}
We examine two counting problems that seem very group-theoretic on the surface but, on closer examination, turn out to concern integers with restrictions on their prime factors.

First, given an odd prime~$q$ and a finite abelian $q$-group $H$, we consider the set of integers~$n\le x$ such that the Sylow $q$-subgroup of the multiplicative group $(\Z/n\Z)^\times$ is isomorphic to~$H$. We show that the counting function of this set of integers is asymptotic to $K x(\log\log x)^\ell/(\log x)^{1/(q-1)}$ for explicit constants~$K$ and~$\ell$ depending on~$q$ and~$H$.

Second, we consider the set of integers $n\le x$ such that the multiplicative group $(\Z/n\Z)^\times$ is ``maximally non-cyclic'', that is, such that all of its prime-power subgroups are elementary groups. We show that the counting function of this set of integers is asymptotic to $A x/(\log x)^{1-\xi}$ for an explicit constant~$A$, where~$\xi$ is Artin's constant.

As it turns out, both of these group-theoretic problems can be reduced to problems of counting integers with restrictions on their prime factors, allowing them to be addressed by classical techniques of analytic number theory.
\end{abstract}

\renewcommand{\labelenumi}{(\alph{enumi})}

\section{Introduction}

Counting problems for integers with restrictions on their prime factors have been a topic of interest to number theorists for many years, with squarefree numbers and friable numbers (integers without large prime factors) being typical examples. Particularly relevant is Landau's investigation~\cite{Landau} of the counting function of those integers~$n$ expressible as the sum of two squares; thanks to the classical characterization of Fermat, this property is equivalent to a restriction on the prime factors of~$n$ that are congruent to~$3\mod4$. Similar ideas allowed Ford, Luca, and Moree~\cite{FordLucaMoree} to count the integers $n\le x$ such that the Euler totient function $\phi(n)$ is not divisible by a fixed prime~$q$, or equivalently such that~$q^2$ does not divide~$n$ and no prime divisor of~$n$ is congruent to~$1\mod q$.  

This last result can be restated as counting the integers~$n$ up to~$x$ for which the Sylow $q$-subgroup of $(\Z/n\Z)^\times$ is the trivial group. And indeed, examining group-theoretic statistics of the family of multiplicative groups $(\Z/n\Z)^\times$ is a fertile source of problems of interest to analytic number theorists, starting directly with the distribution of the cardinality $\phi(n)$. For example, the length of the invariant factor decomposition of $(\Z/n\Z)^\times$ (see Section~\ref{mnc section} for the definition) is essentially the number of distinct prime factors of~$n$, and so these lengths satisfy an Erd\H os--Kac law (they are asymptotically normally distributed when suitably normalized). Recent work of the second author with Chang~\cite{CM} and Troupe~\cite{MartinTroupe} examined, respectively, the counting function of those integers with a prescribed least invariant factor and an Erd\H os--Kac law for the total number of subgroups of~$(\Z/n\Z)^\times$.

When speaking on the work~\cite{MartinTroupe} at the 2017 Alberta Number Theory Days, Lee Troupe was asked by Colin Weir if it was possible to count, for a fixed prime~$q$ and a fixed finite abelian $q$-group~$H$, the number of integers $n\le x$ for which the Sylow $q$-subgroup of $(\Z/n\Z)^\times$ is isomorphic to~$H$, thus generalizing the aforementioned result of Ford, Luca, and Moree. Answering this question is the main focus of this paper; we establish some notation to describe our results.

\begin{definition}
Let $\Z_n=(\Z/n\Z)^+$ and $\Z_n^\times=(\Z/n\Z)^\times$ denote the additive group and multiplicative group, respectively, of the quotient ring $\Z/n\Z$. For any prime~$q$, let $G_q(n)$ denote the Sylow $q$-subgroup of $\Z_n^\times$, that is, the unique subgroup of $\Z_n^\times$ whose cardinality is the highest power of~$q$ that divides~$\phi(n)$.
\end{definition}

\begin{definition} \label{DHx def}
For any finite abelian $q$-group $H$, let $D(H,x)={\#\{n\le x \colon G_q(n)\cong H\}}$.
\end{definition}

Because our asymptotic formula for $D(H,x)$ will depend upon~$H$, we need a standard notation for isomorphism classes of finite abelian $q$-groups.
From the classification of finite abelian groups, every abelian group of prime-power order can be labeled by a partition (a nonincreasing sequence of positive integers).

\begin{definition}
Given a partition $\balpha=(\alpha_1,\dots,\alpha_j)$, denote its length by $\ell(\balpha)=j$. Let $\Z_{q^\balpha}$ denote the finite abelian $q$-group $\Z_{q^{\alpha_1}}\times\Z_{q^{\alpha_2}}\times\dots\times\Z_{q^{\alpha_j}}$.
\end{definition}

The following theorem, which we prove in Section~\ref{DHx section}, gives an asymptotic formula for the number of integers $n\le x$ for which the Sylow $q$-subgroup of $\Z_n^\times$ is isomorphic to any particular finite abelian $q$-group $\Z_{q^\balpha}$.

\begin{theorem} \label{main}
Let $q$ be an odd prime and $\balpha$ a partition. Then
\[
D(\Z_{q^\balpha},x)=K(\Z_{q^\balpha}) \frac{x(\log\log x)^{\ell(\balpha)}}{(\log x)^{1/(q-1)}} \bigg( 1+O_{q,\balpha}\bigg(\frac1{\log\log x} \bigg) \bigg),
\]
where $K(\Z_{q^\balpha})$ is the constant from Definition~\ref{Eqalpha def} below.
\end{theorem}

This result does in fact cover the case where $\balpha$ is the empty partition, so that $\Z_{q^\balpha}$ is the trivial group; consequently, one special case of Theorem~\ref{main} is the result of Ford, Luca, and Moree~\cite{FordLucaMoree} mentioned above.

\begin{corollary} \label{main cor}
Let $q$ be an odd prime. The number of $n\le x$ for which $q\nmid \phi(n)$ equals
\[
D(\Z_{q^{\emptyset}},x) = \frac{B_q x}{(\log x)^{1/(q-1)}} \bigg( 1+O_q \bigg(\frac1{\log\log x} \bigg) \bigg),
\]
where $B_q$ is the constant depending on~$q$ from Definition~\ref{Bq def} below.
\end{corollary}

The methods of this paper could in principle handle the variant of Theorem~\ref{main} where $q=2$, but we do not do so herein. (Note that the analogous variant of Corollary~\ref{main cor} is trivial, since $2\mid\phi(n)$ for all $n\ge3$; this fact would necessitate a somewhat different starting point for the $q=2$ variant of Theorem~\ref{main}, even though the subsequent procedure would be very similar.)

In Section~\ref{mnc section} we shift our focus to a problem involving the global structure of the multiplicative group~$\Z_n^\times$. The motivation for our next theorem comes from the fact that it is easy to count the integers $n\le x$ for which $\Z_n^\times$ is cyclic: these are precisely the integers possessing primitive roots (namely~$1$, $2$, $4$, and $p^r$ and $2p^r$ for odd primes~$p$ and integers $r\ge1$) and thus have counting function asymptotic to $\li(x)+\li(\frac x2) \sim \frac32x/\log x$ by the prime number theorem. We might ask for the opposite extreme: what is the ``least cyclic'' that a finite abelian group can be, and how prevalent are such groups in the family of multiplicative groups?

Several notions of such a group being ``maximally non-cyclic'' turn out to be mutually equivalent (see Definition~\ref{mncdef} below); one way to describe a maximally non-cyclic finite abelian group is one all of whose Sylow $q$-subgroups are elementary groups (direct products of copies of~$\Z_q$). The counting function for the corresponding integers~$n$ turns out to be quite interesting:

\begin{theorem} \label{mnc theorem}
The number of integers $n$ up to $x$ such that $\Z_n^\times$ is maximally non-cyclic is asymptotic to $Ax/(\log x)^{1-\xi}$, where~$\xi$ is Artin's constant and~$A$ is the constant in Definition~\ref{Artin} below.
\end{theorem}

\noindent The proof gives a relative error of $1/(\log x)^{1-\ep}$ for any~$\ep>0$. We point out that the main term contains an exponent of $\log x$ that is (presumably) irrational, which is an unusual feature of an asymptotic formula arising from a reasonably natural property of~$\Z_n^\times$.

\section{Multiplicative groups with a prescribed Sylow $q$-subgroup} \label{DHx section}

Throughout this paper, we will use the letters~$p$ (with or without subscripts), $q$, and~$t$ exclusively to denote primes.

\subsection{Reduction to analytic number theory}

In this section we convert the group-theoretic property $G_q(n)=H$ into specific constraints on the primes dividing~$n$, so that the problem of counting multiplicative groups~$\Z_n^\times$ with specified Sylow $q$-subgroup~$H$ is converted into a nested sum indexed by a particular factorization of~$n$ (see Proposition~\ref{initial D0 prop}). Since the power of~$q$ itself that divides~$n$ affects the Sylow $q$-subgroup of $\Z_n^\times$ in a particular way, we stratify the integers according to that power and count each stratum of integers separately.

\begin{definition}
For a prime~$q$ and a nonzero integer~$x$, define $\nu_q(x)$ to be the largest nonnegative integer $k$ such that $q^k$ divides~$x$.
\end{definition}

\begin{definition} \label{DkHx def}
Given a prime~$q$, a finite abelian $q$-group~$H$, and a nonnegative integer~$k$, define $D_k(H,x)=\#\{n\le x \colon \nu_q(n)=k,\, G_q(n)=H\}$. Note that $D(H,x) = \sum_{k=0}^\infty D_k(H,x)$, where $D(H,x)$ is as in Definition~\ref{DHx def}.
\end{definition}

\noindent
We will show in Lemma~\ref{handle Dk lemma} that $D_k(H,x)$, if nonzero, is equal to $D_0(\Z_{H'},x)$ for a certain $q$-group~$H'$ depending on~$H$ and~$k$. Therefore our main technical goal in Section~\ref{DHx section} is to obtain an asymptotic formula for $D_0(H,x)$.

To avoid requiring notation for the lengths of partitions, we can regard partitions as infinite nonincreasing sequences of nonnegative integers with only finitely many positive terms.

\begin{definition} \label{partition def}
For any partition $\balpha = (\alpha_1,\alpha_2,\ldots)$, the conjugate partition $\ba = (a_1,a_2,\ldots)$ of $(\alpha_1,\alpha_2,\ldots)$ is the partition whose Ferrers diagram is the transpose of the Ferrers diagram of $(\alpha_1,\alpha_2,\ldots)$, so that $a_j=\#\{k\colon \alpha_k\ge j\}$ for all $j\ge1$; in particular, $a_j-a_{j+1} = \#\{k\colon \alpha_k=j\}$. Using this conjugate partition, define
\begin{equation} \label{Calpha def}
C(\balpha) = \prod_{u=1}^\infty \frac1{(a_u-a_{u+1})!} = \prod_{u=1}^{\alpha_1} \frac1{(a_u-a_{u+1})!}.
\end{equation}
\end{definition}

\begin{proposition} \label{initial D0 prop}
For any odd prime~$q$ and any partition $\balpha=(\alpha_1,\dots,\alpha_j)$,
\begin{align}
D_0 & (\Z_{q^\balpha},x) \label{D_0} \\
&= C(\balpha) \sum_{\substack{p_1\le x/3\\\nu_q(p_1-1)=\alpha_1}} \sum_{\substack{p_2\le x/3p_1\\p_2\ne p_1\\\nu_q(p_2-1)=\alpha_2}} \cdots \sum_{\substack{p_{j-1}\le x/3p_1\cdots p_{j-2} \\ p_{j-1}\ne p_1,\ldots,p_{j-2} \\ \nu_q(p_{j-1}-1)=\alpha_{j-1}}} \sum_{\substack{p_j\le x/p_1\cdots p_{j-1}\\p_j\ne p_1,\ldots,p_{j-1}\\\nu_q(p_j-1)=\alpha_j}} \sum_{\substack{m\le x/p_1\cdots p_j\\q\dnd m\\ (t\mid m \text{ and }t\equiv1\mod q)\Rightarrow t\in\{p_1,\dots,p_j\}}} 1, \notag
\end{align}
where~$t$ denotes a generic prime factor of~$m$. Here, $C(\balpha)$ is the constant defined in equation~\eqref{Calpha def}.
\end{proposition}

\begin{proof}
We can write~$n$ as the product of primes $n=2^{\beta}p_1^{\beta_1}p_2^{\beta_2}\cdots p_k^{\beta_k}$ where $\beta\ge 0$, $\beta_1,\beta_2,\ldots,\beta_k>0$ and $q\ne p_j$ for each $1\le j\le k$. By the Chinese remainder theorem
\begin{align*}
\Z_n^\times &\cong \Z_{2^\beta}^\times\times \big( \Z_{p_1^{\beta_1-1}}\times\Z_{p_1-1} \big) \times \big( \Z_{p_2^{\beta_2-1}}\times\Z_{p_2-1} \big) \times\cdots\times \big( \Z_{p_k^{\beta_t-1}}\times\Z_{p_k-1} \big) \\
&\cong \big( \Z_{2^\beta}^\times \times \Z_{p_1^{\beta_1-1}} \times \Z_{p_2^{\beta_2-1}} \times\cdots\times \Z_{p_t^{\beta_k-1}} \big) \times \big( \Z_{p_1-1} \times \Z_{p_2-1} \times\cdots\times \Z_{p_k-1} \big).
\end{align*}
Since $q$ is an odd prime not dividing~$n$, we see that~$q$ does not divide the cardinality of the first factor; therefore the Sylow $q$-subgroup of $\Z_n^\times$ is the same as the Sylow $q$-subgroup of the second factor, which is simply $\Z_{q^{\nu_q(p_1-1)}} \times \Z_{q^{\nu_q(p_2-1)}} \times\cdots\times \Z_{q^{\nu_q(p_k-1)}}$.

It follows that $G_q(n)=\Z_{q^\balpha}$ if and only if the multisets $\balpha$ and $\{\nu_q(p_i-1)\colon 1\le i\le k\}$ are the same except for occurrences of~$0$ in the latter multiset, that is, if and only if~$n$ has, for every integer $u\ge1$, exactly $a_u-a_{u+1}$ distinct prime factors $p$ satisfying $\nu_q(p)=\alpha_u$.

Now set $j=\ell(\balpha)$ and consider the expression
\begin{equation} \label{pre Calpha}
\sum_{\substack{p_1\le x\\\nu_q(p_1-1)=\alpha_1}}\sum_{\substack{p_2\le x/p_1\\p_2\ne p_1\\\nu_q(p_2-1)=\alpha_2}}\cdots \sum_{\substack{p_j\le x/p_1\cdots p_{j-1}\\p_j\ne p_1,\ldots,p_{j-1}\\\nu_q(p_j-1)=\alpha_j}}\sum_{\substack{m\le x/p_1\cdots p_j\\q\dnd m\\ (t\mid m \text{ and }t\equiv1\mod q)\Rightarrow t\in\{p_1,\dots,p_j\}}}1.
\end{equation}
This expression counts integers of the form $n=p_1p_2\cdots p_jm$, where $p_1,\ldots,p_j$ are distinct primes such that $\{\nu_q(p_i-1)\colon 1\le i\le j\} = \balpha$ as multisets and $\nu_q(p-1)=0$ for every $p\mid m$. In other words, it counts integers~$n$ such that $G_q(n)=\Z_{q^\balpha}$, except that it counts such integer with multiplicity because $(p_1,\dots,p_j)$ is an ordered tuple: for each $u\ge1$ we may arbitrarily permute the $a_u-a_{u+1}$ primes $p_i$ in the tuple that satisfy $\nu_q(p_i)=\alpha_u$ and still obtain the same~$n$. Consequently we must divide by $(a_u-a_{u+1})!$ for each $u\ge1$ to compensate for this multiple counting, which is the same as multiplying the expression~\eqref{pre Calpha} by~$C(\balpha)$.

Finally, if $p_i$ is greater than $x/3p_1\cdots p_{i-1}$ for any $1\le i\le j-1$, then the sum over $p_{i+1}$ in the expression~\eqref{pre Calpha} is empty, and therefore we may alter the upper bounds of summation accordingly, reaching the expression in equation~\eqref{D_0} as desired.
\end{proof}

\subsection{Application of the Selberg--Delange method}

Proposition~\ref{initial D0 prop} provides a clear relationship between the original problem of counting prescribed Sylow $q$-subgroups and the more analytic problem of counting integers with restrictions on their prime factors. The innermost sum in equation~\eqref{D_0}, in particular, is exactly of this latter type, and thus can be successfully estimated by the Selberg--Delange method. We cite an application of this technique from~\cite{CM} that has been tailored to this purpose.

\begin{definition} \label{Bq def}
For any odd prime $q$ and any prime $p\ne q$, let~$k_p$ denote the multiplicative order of~$p$ modulo~$q$. Then define
\[
B_q=\frac{1}{\Gamma(1-1/(q-1))}\bigg(1-\frac1q \bigg)^{1-1/(q-1)}\prod_{\substack{p\ne q\\ p\not\equiv 1 \mod q}} \bigg(1-\frac1{p^{k_p}} \bigg)^{-1/k_p}\prod_{\chi\ne \chi_0}L(1,\chi)^{-1/(q-1)},
\]
where $\Gamma(z)$ is the classical Gamma function.
\end{definition}

\begin{proposition} \label{sdsum2}
Let $q$ be an odd prime, and let $p_1,\ldots,p_j$ be distinct prime numbers congruent to $1\mod q$. For $y\ge3$,
\begin{align*}
\sum_{\substack{m\le y\\ q\nmid m\\ (t\mid m \text{ and }t\equiv 1\mod q)\Rightarrow t\in\{p_1,\ldots,p_j\} }}1 &=\frac{B_q y}{(\log y)^{1/(q-1)}}\prod_{i=1}^j\bigg(1-\frac 1{p_i}\bigg)^{-1} + O_j\bigg( \frac y{(\log y)^{1+1/(q-1)}} \bigg),
\end{align*}
where $t$ denotes a generic prime factor of~$m$.
\end{proposition}

\begin{proof}
Define unions of residue classes
\begin{align*}
\B = \{n\not\equiv0,1\mod q\} \quad\text{and}\quad \B' = \{n\equiv1\mod q\},
\end{align*}
so that we are trying to count integers whose prime factors all lie in~$\B \cup \{p_1,\dots,p_j\}$.

We begin by quoting~\cite[Theorem~3.6]{CM} with the set $\B$ just defined, so that (using Notation~3.1 from that paper) $B=q-2$ and $\underline B=1$ and $\beta = 1-1/(q-1)$; we also set $\mathcal I=\{p_1,\dots,p_j\}$ and $\mathcal R=\emptyset$. Since we are allowing our error terms to depend on~$q$, we may simplify the error term from~\cite[Theorem~3.6]{CM}, and we may also ignore the assumption that $\log y \ge \alpha q^{1/2}\log^2 q$. The conclusion (remembering that $t$ denotes a generic prime) is that
\begin{multline}
\#\big\{ m\le y\colon q\nmid m,\, (t\mid m \text{ and }t\equiv 1\mod q) \Rightarrow t\in\{p_1,\ldots,p_j\} \big\} \\
= \frac y{(\log y)^{1/(q-1)}} \bigg( \frac{G_\B(1)}{\Gamma(1-1/(q-1))} \prod_{i=1}^j \bigg( 1-\frac1{p_i} \bigg)^{-1} + O\big( 2^j (\log y)^{-1} \big) \bigg). \label{together Gamma}
\end{multline}
It thus remains to evaluate $G_\B(1)$.

In the proof of~\cite[Proposition~4.1]{CM}, where the set $\B'$ is denoted as $\{1\}$, it is shown that
\begin{equation*}
G_{\B'}(1) = \bigg( \frac{\phi(q)}q \prod_{\substack{\chi\mod q \\ \chi \ne \chi_0}} L(1,\chi) \bigg)^{1/\phi(q)} \prod_{\substack{p \nmid q \\ p \not\equiv 1 \mod{q}}} \bigg(1-\frac{1}{p^{\ord_q(p)}} \bigg)^{1/\ord_q(p)}.
\end{equation*}
Moreover,~\cite[Remark~3.5]{CM} tells us that $G_\B(1) G_{\B'}(1) = \phi(q)/q$. Therefore
\begin{align*}
G_\B(1) = \frac{\phi(q)}{qG_{\B'}(1)} &= \frac{\phi(q)}q \bigg( \frac{\phi(q)}q \prod_{\substack{\chi\mod q \\ \chi \ne \chi_0}} L(1,\chi) \bigg)^{-1/\phi(q)} \prod_{\substack{p \nmid q \\ p \not\equiv 1 \mod{q}}} \bigg(1-\frac{1}{p^{\ord_q(p)}} \bigg)^{-1/\ord_q(p)} \\
&= \bigg( 1-\frac1q \bigg)^{1-1/(q-1)} \prod_{\substack{\chi\mod q \\ \chi \ne \chi_0}} L(1,\chi)^{-1/(q-1)} \prod_{\substack{p \ne q \\ p \not\equiv 1 \mod{q}}} \bigg(1-\frac{1}{p^{\ord_q(p)}} \bigg)^{-1/\ord_q(p)}
\end{align*}
which, together with the Gamma factor from equation~\eqref{together Gamma}, equals~$B_q$ as given in Definition~\ref{Bq def}.
\end{proof}

\subsection{Technical lemmas}

Motivated by the expressions that will appear when we apply Proposition~\ref{sdsum2} to equation~\eqref{D_0}, we now establish a collection of technical lemmas that will be used in the next section to prove the important Proposition~\ref{recsum}. That result will subsequently allow us to establish the recursive Propositions~\ref{recarg} and~\ref{recarg2}, which will provide an evaluation of the iterated sum in equation~\eqref{D_0}. Though many of the techniques of this section are standard, we do highlight the use of the following hypergeometric function as a tool for evaluating certain sums over primes with fractional powers of a logarithm (see Lemmas~\ref{lsum} and~\ref{usum}).

\begin{definition} \label{H def}
For any $\gamma\in\R\setminus\N$, define
\[
H_\gamma(z) = -\sum_{n=1}^\infty \frac \gamma{n-\gamma}z^n.
\]
Note that the power series defining $H_\gamma (z)$ converges for $|z|<1$ by the ratio test. (One could also define $H_\gamma(z) = -\gamma\Phi(z,1,-\gamma)-1$ where $\Phi$ is the Hurwitz--Lerch transcendent.)
\end{definition}

\begin{lemma} \label{lint lemma}
Let $\gamma>0$ such that $\gamma\not\in\N$.
\begin{enumerate}
\item For $0\le z\le\frac12$, we have $H_\gamma(z) \ll_\gamma z$.
\item For $0\le z<1$, we have $H_\gamma(z) = \gamma\log(1-z) + O_\gamma(1)$.
\end{enumerate}
\end{lemma}

\begin{proof}
Part (a) follows simply from the fact that $H_\gamma$ is analytic on a neighborhood of $[0,\frac12]$ and $H_\gamma(0)=0$.
Using Definition~\ref{H def} and the power series for $\log(1-z)$,
\begin{align*}
H_\gamma(z)-\gamma\log(1-z)&=-\sum_{n=1}^{\infty} \frac{\gamma z^n}{n-\gamma} - \gamma\sum_{n=1}^\infty \frac{-z^n}n \\
&= -\gamma^2\sum_{n=1}^\infty \frac{z^n}{n(n-\gamma)} \ll_\gamma \sum_{n=1}^\infty \frac{1}{n|n-\gamma|} \ll_\gamma 1,
\end{align*}
which establishes part~(b).
\end{proof}

\begin{lemma} \label{indefint lemma}
Let $\gamma>0$ such that $\gamma\not\in\N$, and let $x>1$. For $0<z<1$,
\[
\frac d{dz}\bigg( \frac{H_\gamma(z)}{\gamma(z\log x)^\gamma} \bigg) = -\frac 1{(1-z)(z\log x)^\gamma}.
\]
\end{lemma}

\begin{proof}
We differentiate the power series in Definition~\ref{H def} term by term to obtain
\begin{align*}
\frac d{dz}\bigg( \frac{H_\gamma(z)}{\gamma(z\log x)^\gamma} \bigg)&=\frac d{dz}\bigg( \frac{-\sum_{n=1}^\infty \frac \gamma{n-\gamma}z^n}{\gamma(z\log x)^\gamma} \bigg)\\
&=\frac{-\gamma(z\log x)^\gamma\sum_{n=1}^\infty \frac {n\gamma}{n-\gamma}z^{n-1}+\gamma^2(\log x)^\gamma z^{\gamma-1}\sum_{n=1}^\infty \frac \gamma{n-\gamma}z^n}{\gamma^2(z\log x)^{2\gamma}}\\
&=\frac{-\gamma^2(\log x)^\gamma z^\gamma\big( \sum_{n=1}^\infty \frac {n}{n-\gamma}z^{n-1}-\sum_{n=1}^\infty \frac \gamma{n-\gamma}z^{n-1} \big)}{\gamma^2(z\log x)^{2\gamma}}\\
&=\frac{-\sum_{n=1}^\infty z^{n-1}}{(z\log x)^\gamma} = -\frac 1{(1-z)(z\log x)^\gamma}
\end{align*}
as desired.
\end{proof}

\begin{lemma} \label{indefint}
Let $\gamma>0$ such that $\gamma\not\in\N$, and let $x>1$. For $1<u<x$,
\[
\frac d{du}\bigg( \frac{H_\gamma\big(1-\frac{\log u}{\log x}\big)}{\gamma(\log(x/u))^\gamma} \bigg) = \frac 1{(u\log u)(\log(x/u))^\gamma}.
\]
\end{lemma}

\begin{proof}
Using the change of variables $z=1-(\log u)/\log x$, so that $z\log x=\log(x/u)$ and $\frac{dz}{du} = -1/(u\log x)$, we have
\[
\frac d{du}\bigg( \frac{H_\gamma\big(1-\frac{\log u}{\log x}\big)}{\gamma(\log(x/u))^\gamma} \bigg) = \frac d{dz}\bigg( \frac{H_\gamma(z)}{\gamma(z\log x)^\gamma} \bigg) \bigg( {-}\frac 1{u\log x} \bigg).
\]
The assumption $1<u<x$ implies that $0<z<1$, and so by Lemma~\ref{indefint lemma},
\begin{align*}
\frac d{du}\bigg( \frac{H_\gamma\big(1-\frac{\log u}{\log x}\big)}{\gamma(\log(x/u))^\gamma} \bigg) &= -\frac 1{(1-z)(z\log x)^\gamma} \bigg( {-}\frac 1{u\log x} \bigg) \\
&= \frac 1{((\log u)/\log x)(\log(x/u))^\gamma} \frac 1{u\log x}
\end{align*}
as desired.
\end{proof}

\begin{lemma} \label{lint}
Let $\gamma>0$. For $y\ge 4$,
\[
\int_2^{\sqrt y} \frac 1{(u\log u)(\log(y/u))^\gamma} \,du = \frac{\log\log y}{(\log y)^\gamma}+O_\gamma\bigg( \frac 1{(\log y)^\gamma}\bigg).
\]
\end{lemma}

\begin{proof}
Since $u\le\sqrt y$ in the integrand, we may write
\begin{align*}
(\log(y/u))^{-\gamma} &= (\log y)^{-\gamma} \bigg( 1-\frac{\log u}{\log y} \bigg)^{-\gamma} \\
&= (\log y)^{-\gamma} \bigg( 1 + O_\gamma\bigg( \frac{\log u}{\log y} \bigg) \bigg) = (\log y)^{-\gamma} + O_\gamma \big( (\log y)^{-\gamma-1}\log u \big),
\end{align*}
and therefore
\begin{align*}
\int_2^{\sqrt y} & \frac 1{(u\log u)(\log(y/u))^\gamma} \,du \\
&= (\log y)^{-\gamma} \int_2^{\sqrt y} \frac 1{u\log u} \,du + O_\gamma\bigg( (\log y)^{-\gamma-1} \int_2^{\sqrt y} \frac{\log u}{u\log u} \,du \bigg) \\
&= (\log y)^{-\gamma} (\log\log \sqrt y-\log\log2) + O_\gamma\big( (\log y)^{-\gamma-1} (\log\sqrt y-\log 2) \big)
\end{align*}
which implies the statement of the lemma since $\log\log\sqrt y=\log\log y+O(1)$.
\end{proof}

\begin{lemma} \label{lsum}
Let $\gamma>0$ such that $\gamma\not\in\N$, let~$q$ be prime, and let $\alpha\in\N$. For $y\ge 4$,
\[
\sum_{\substack{p\le \sqrt y\\ \nu_q(p-1)=\alpha}} \frac1{p(\log(y/p))^\gamma} = \frac{\log\log y}{q^\alpha(\log y)^\gamma}+O_\gamma\bigg( \frac 1{(\log y)^\gamma}\bigg).
\]
\end{lemma}

\begin{proof}
If we define
\begin{equation} \label{Mx def}
M(x)=\sum_{\substack{p\le x\\\nu_q(p-1)=\alpha}} \frac1p = \sum_{\substack{p\le x\\p\equiv 1\mod{q^\alpha}}} \frac1p - \sum_{\substack{p\le x\\p\equiv 1\mod {q^{\alpha+1}}}} \frac1p,
\end{equation}
then a Mertens-type formula for arithmetic progressions~\cite[Corollary~4.12]{MV} shows that there exist constants $c_{q^\alpha}$ and $c_{q^{\alpha+1}}$ such that
\begin{align*}
M(x) &= \bigg( \frac{\log\log x}{\phi(q^\alpha)} + c_{q^\alpha} + O\bigg( \frac1{\log x} \bigg) \bigg) - \bigg( \frac{\log\log x}{\phi(q^{\alpha+1})} + c_{q^{\alpha+1}} + O\bigg( \frac1{\log x} \bigg) \bigg) \\
&= \frac{\log\log x}{q^\alpha} + c_{q^\alpha} - c_{q^{\alpha+1}} + O\bigg( \frac1{\log x} \bigg)
\end{align*}
for $x\ge2$. Setting
\begin{equation} \label{Rx def}
R(x) = M(x) - \bigg( \frac{\log\log x}{q^\alpha} + c_{q^\alpha} - c_{q^{\alpha+1}} \bigg) \ll \frac1{\log x},
\end{equation}
it follows that
\begin{align}
\sum_{\substack{p\le \sqrt y\\ \nu_q(p-1)=\alpha}} \frac1{p(\log(y/p))^\gamma} &= \int_2^{\sqrt y} \frac1{(\log(y/u))^\gamma} \,dM(u) \notag \\
&=\int_2^{\sqrt y} \frac1{(\log(y/u))^\gamma} \,d\bigg( \frac{\log\log u}{q^\alpha} + c_{q^\alpha} - c_{q^{\alpha+1}} + R(u) \bigg) \notag \\
&=\frac 1{q^\alpha} \int_2^{\sqrt y} \frac1{(\log(y/u))^\gamma} \frac {du}{u\log u} + \int_2^{\sqrt y} \frac1{(\log(y/u))^\gamma} \,dR(u) \notag \\
&=\frac 1{q^\alpha} \bigg( \frac{\log\log y}{(\log y)^\gamma}+O_\gamma\bigg( \frac1{(\log y)^\gamma}\bigg) \bigg) + \int_2^{\sqrt y} \frac1{(\log(y/u))^\gamma} \,dR(u) \label{RS int}
\end{align}
by Lemma~\ref{lint}.
On the other hand, integrating by parts yields
\begin{align*}
\int_2^{\sqrt y} \frac1{(\log(y/u))^\gamma} \,dR(u) &= \frac{R(u)}{(\log(y/u))^\gamma} \bigg|_2^{\sqrt y} - \int_2^{\sqrt y} R(u) \frac d{du} \bigg( \frac1{(\log(y/u))^\gamma} \bigg) \,du \\
&= \frac{R(\sqrt y)}{(\log\sqrt y)^\gamma} - \frac{R(2)}{(\log(y/2))^\gamma} - \int_2^{\sqrt y} \frac{\gamma R(u)}{u(\log(y/u))^{\gamma+1}} \,du \\
&\ll_\gamma \frac{1/\log y}{(\log y)^\gamma} + \frac1{(\log y)^\gamma} + \int_2^{\sqrt y} \frac{1/\log u}{u(\log(y/u))^{\gamma+1}} \,du \\
&\ll_\gamma \frac1{(\log y)^\gamma} + \frac1{(\log y)^{\gamma+1}} \int_2^{\sqrt y} \frac1{u\log u} \,du \\
&\ll \frac1{(\log y)^\gamma} + \frac{\log\log \sqrt y- \log\log 2}{(\log y)^{\gamma+1}},
\end{align*}
which, combined with equation~\eqref{RS int}, establishes the lemma.
\end{proof}

\begin{lemma} \label{uint}
Let $\gamma>0$ such that $\gamma\not\in\N$. For $y\ge 9$,
\[
\int_{\sqrt y}^{y/3} \frac 1{(u\log u)(\log(y/u))^\gamma} \,du \ll_\gamma \frac 1{(\log y)^{\min\{\gamma,1\}}}.
\]
\end{lemma}

\begin{proof}
By Lemma~\ref{indefint},
\begin{align*}
\int_{\sqrt y}^{y/3} \frac 1{(u\log u)(\log(y/u))^\gamma} \,du &= \frac{H_\gamma(1-\frac{\log u}{\log y})}{\gamma(\log(y/u))^\gamma}\bigg|_{\sqrt y}^{y/3} \\
&= \frac{H_\gamma(\frac{\log3}{\log y})}{\gamma(\log3)^\gamma} - \frac{H_\gamma(1/2)}{\gamma\big( \frac 12 \log y\big)^\gamma} \\
&\ll_\gamma \bigg| H_\gamma\bigg( \frac{\log3}{\log y} \bigg) \bigg| + \frac1{(\log y)^\gamma} \ll_\gamma \frac{\log3}{\log y} + \frac1{(\log y)^\gamma}
\end{align*}
by Lemma~\ref{lint lemma}(a); this bound is equivalent to the statement of the lemma.
\end{proof}

\begin{lemma} \label{usum}
Let $\gamma>0$ such that $\gamma\not\in\N$, let~$q$ be prime, and let $\alpha\in\N$. For $y\ge9$,
\[
\sum_{\substack{\sqrt y<p\le y/3\\ \nu_q(p-1)=\alpha}} \frac1{p(\log(y/p))^\gamma} \ll_\gamma \frac1{(\log y)^{\min\{\gamma,1\}}}.
\]
\end{lemma}

\begin{proof}
With $M(x)$ and $R(x)$ defined as in equations~\eqref{Mx def} and~\eqref{Rx def},
\begin{align}
\sum_{\substack{\sqrt y<p\le y/3\\ \nu_q(p-1)=\alpha}} \frac1{p(\log(y/p))^\gamma} &= \int_{\sqrt y}^{y/3} \frac1{(\log(y/u))^\gamma} \,dM(u) \notag \\
&=\int_{\sqrt y}^{y/3} \frac1{(\log(y/u))^\gamma} \,d\bigg( \frac{\log\log u}{q^\alpha} + c_{q^\alpha} - c_{q^{\alpha+1}} + R(u) \bigg) \notag \\
&=\frac 1{q^\alpha} \int_{\sqrt y}^{y/3} \frac1{(\log(y/u))^\gamma} \frac {du}{u\log u} + \int_{\sqrt y}^{y/3} \frac1{(\log(y/u))^\gamma} \,dR(u) \notag \\
&\ll_\gamma \frac 1{(\log y)^{\min\{\gamma,1\}}} + \int_{\sqrt y}^{y/3} \frac1{(\log(y/u))^\gamma} \,dR(u) \label{almost a copy}
\end{align}
by Lemma~\ref{uint}.
On the other hand, integrating by parts yields
\begin{align*}
\int_{\sqrt y}^{y/3} \frac1{(\log(y/u))^\gamma} \,dR(u) &= \frac{R(u)}{(\log(y/u))^\gamma} \bigg|_{\sqrt y}^{y/3} - \int_{\sqrt y}^{y/3} R(u) \frac d{du} \bigg( \frac1{(\log(y/u))^\gamma} \bigg) \,du \\
&= \frac{R(y/3)}{(\log3)^\gamma} - \frac{R(\sqrt y)}{(\log\sqrt y)^\gamma} - \int_{\sqrt y}^{y/3} \frac{\gamma R(u)}{u(\log(y/u))^{\gamma+1}} \,du \\
&\ll_\gamma \frac1{\log y} + \frac{1/\log y}{(\log y)^\gamma} + \int_{\sqrt y}^{y/3} \frac{1/\log u}{u(\log(y/u))^{\gamma+1}} \,du \\
&\ll_\gamma \frac1{\log y} + \frac1{\log y} \int_{\sqrt y}^{y/3} \frac1{(\log(y/u))^{\gamma+1}} \,\frac{du}u.
\end{align*}
Using the change of variables $v=y/u$ yields
\begin{align*}
\int_{\sqrt y}^{y/3} \frac1{(\log(y/u))^\gamma} \,dR(u) &\ll_\gamma \frac1{\log y} + \frac1{\log y} \int_3^{\sqrt y} \frac1{(\log v)^{\gamma+1}} \,\frac{dv}v \\
&= \frac1{\log y} + \frac1{\log y} \frac{-\gamma}{(\log v)^\gamma} \bigg|_3^{\sqrt y} \\
&\ll_\gamma \frac1{\log y} + \frac1{\log y} \bigg( 1 - \frac1{(\log y)^\gamma} \bigg) \ll \frac1{\log y}
\end{align*}\
which, combined with equation~\eqref{almost a copy}, establishes the lemma.
\end{proof}

\begin{lemma} \label{sumlem}
For $y\ge 9$,
\[
\sum_{\sqrt y<p\le y/3} \frac1{p\log(y/p)} \ll \frac{\log\log y}{\log y}.
\]
\end{lemma}

\begin{proof}
We first consider, for $2\le U\le y^{3/4}$,
\[
\sum_{y/2U \le p < y/U} \frac1{p\log(y/p)} \ll \frac{\pi(y/U)}{(y/U) \log U} \ll \frac1{\log(y/U)\log U} \ll \frac1{\log y\cdot \log U}.
\]
Applying this estimate with $U=2,4,8,\dots$ until $U$ amply exceeds $\sqrt y$, we deduce that
\[
\sum_{y/2U \le p < y/U} \frac1{p\log(y/p)} \le \sum_{k=1}^{\log y} \sum_{y/2^{k+1} \le p < y/2^k} \frac1{p\log(y/p)} \ll \sum_{k=1}^{\log y} \frac1{\log y\cdot \log 2^k} \ll \frac{\log\log y}{\log y}. \qedhere
\]
\end{proof}

\begin{lemma} \label{sum}
Let $\gamma>0$ such that $\gamma\not\in\N$, let~$q$ be prime, and let $\alpha\in\N$. For $y\ge 3$,
\[
\sum_{\substack{p\le y/3 \\ \nu_q(p-1)=\alpha}} \frac1{p(\log(y/p))^\gamma} = \frac{\log\log y}{q^\alpha(\log y)^\gamma}+O_\gamma\bigg( \frac 1{(\log y)^{\min\{\gamma,1\}}} \bigg).
\]
\end{lemma}

\begin{proof}
Since the sum in question is empty when $3\le y<9$, we may assume that $y\ge9$.
By Lemmas~\ref{lsum} and~\ref{usum},
\begin{align*}
\sum_{\substack{p\le y/3\\ \nu_q(p-1)=\alpha}} \frac1{p(\log(y/p))^\gamma} &= \sum_{\substack{p\le \sqrt y\\ \nu_q(p-1)=\alpha}} \frac1{p(\log(y/p))^\gamma} + \sum_{\substack{\sqrt y<p\le y/3\\ \nu_q(p-1)=\alpha}} \frac1{p(\log(y/p))^\gamma} \\
&= \frac{\log\log y}{q^\alpha(\log y)^\gamma} + O_\gamma\bigg( \frac 1{(\log y)^\gamma}\bigg) + O_\gamma\bigg( \frac1{(\log y)^{\min\{\gamma,1\}}} \bigg)
\end{align*}
as desired.
\end{proof}

\begin{corollary} \label{sumcor}
Let $0<\gamma\le1$, let~$q$ be prime, and let $\alpha\in\N$. Then, uniformly for $y\ge 9$,
\[
\sum_{\substack{p\le y/3 \\ \nu_q(p-1)=\alpha}} \frac1{p(\log(y/p))^\gamma} \ll_\gamma \frac{\log\log y}{(\log y)^\gamma}.
\]
\end{corollary}

\begin{proof}
For $0<\gamma<1$ this estimate follows immediately from Lemma~\ref{sum}, while for $\gamma=1$ it follows from Lemmas~\ref{lsum} and~\ref{sumlem}.
\end{proof}

We need only two more lemmas of this flavour before we begin to evaluate the inner sums in equation~\eqref{D_0} in the next section.

\begin{lemma} \label{recsum lemma}
Let $k\ge0$ and $\gamma>0$ be real numbers such that $\gamma\not\in\N$, let~$q$ be prime, and let $\alpha\in\N$. For $y\ge 9$,
\[
\sum_{\substack{p\le y/3 \\ \nu_q(p-1)=\alpha}}\frac{(\log\log(y/p))^k}{p(\log(y/p))^\gamma} = \frac{(\log\log y)^{k+1}}{q^\alpha\log^\gamma{y}}+O_{k,\gamma}\bigg( \frac{(\log\log y)^k}{(\log y)^{\min\{\gamma,1\}}}\bigg).
\]
\end{lemma}

\begin{proof}
The upper part of the range of summation can be addressed by noting that
\[
\sum_{\substack{\sqrt y<p\le y/3\\ \nu_q(p-1)=\alpha}}\frac{(\log\log(y/p))^k}{p(\log(y/p))^\gamma} \le (\log\log y)^k\sum_{\substack{\sqrt y<p\le y/3\\ \nu_q(p-1)=\alpha}}\frac{1}{p(\log(y/p))^\gamma} \ll_\gamma \frac{(\log\log y)^k}{(\log y)^{\min\{\gamma,1\}}}
\]
by Lemma~\ref{usum}. As for the remainder of the range of summation,
since $\log\log(y/p) = \log\log y + O(1)$ for $p\le \sqrt y$, we have
\begin{align*}
\sum_{\substack{p\le \sqrt y\\ \nu_q(p-1)=\alpha}}\frac{(\log\log(y/p))^k}{p(\log(y/p))^\gamma} &= \big( (\log\log y)^k+O_k\big( (\log\log y)^{k-1} \big) \big) \sum_{\substack{p\le \sqrt y\\ \nu_q(p-1)=\alpha}}\frac{1}{p(\log(y/p))^\gamma} \\
&= \big( (\log\log y)^k+O_k\big( (\log\log y)^{k-1} \big) \big) \bigg( \frac{\log\log y}{q^\alpha(\log y)^\gamma} + O_\gamma\bigg( \frac 1{(\log y)^\gamma} \bigg) \bigg)
\end{align*}
by Lemma~\ref{lsum}. Combining these two estimates establishes the lemma.
\end{proof}

\begin{lemma} \label{sumsq}
Let $\gamma\in\R$, let~$q$ be prime, and let $\alpha\in\N$. Then, uniformly for $y\ge3$,
\[
\sum_{\substack{p\le y/3 \\ \nu_q(p-1)=\alpha}} \frac1{p^2(\log(y/p))^\gamma} \ll_\gamma (\log y)^{-\gamma}.
\]
\end{lemma}

\begin{proof}
Indeed, the desired bound holds even if we ignore the condition $\nu_q(p-1)=\alpha$, since
\begin{align*}
\sum_{\substack{p\le y/3\\ \nu_q(p-1)=\alpha}} \frac1{p^2(\log(y/p))^\gamma} &\le \sum_{p\le\sqrt y} \frac1{p^2(\log(y/p))^\gamma} + \sum_{p\le\sqrt y} \frac1{p^2(\log(y/p))^\gamma} \\
&\le \sum_{p\le \sqrt y} \frac1{p^2(\log\sqrt y)^\gamma} + \sum_{n>\sqrt y} \frac1{n^2(\log3)^\gamma} \\
&\ll_\gamma (\log y)^{-\gamma} \sum_p \frac1{p^2} + \frac1{\sqrt y} \ll (\log y)^{-\gamma}.
\qedhere
\end{align*}
\end{proof}

\subsection{Recursive evaluation of iterated sums}

The technical lemmas in the previous section hint at the types of expressions that will appear as we work our way through the nested sums in equation~\eqref{D_0}. In this section we establish the results that allow us to recursively evaluate these expressions asymptotically.

\begin{proposition} \label{recsum}
Let $k\ge0$ and $\gamma>0$ be real numbers such that $\gamma\not\in\N$, let $\{w_1,w_2,\ldots,w_n\}$ be a set of $n$ distinct primes, let~$q$ be prime, and let $\alpha\in\N$. For $y\ge 3$,
\begin{multline}
\sum_{\substack{p\le y/3\\ p\ne w_1,w_2,\ldots, w_n \\ \nu_q(p-1)=\alpha}} \frac{p+O(1)}{p^2} \bigg(\frac{(\log\log(y/p))^k}{(\log(y/p))^\gamma}+O\bigg( \frac {(\log\log y)^{k-1}}{(\log(y/p))^{\min\{\gamma,1\}}}\bigg) \bigg)\\
=\frac{(\log\log y)^{k+1}}{q^\alpha\log^\gamma{y}}+O_{n,k,\gamma} \bigg( \frac{(\log\log y)^k}{(\log y)^{\min\{\gamma,1\}}} \bigg).
\label{recsum eq}
\end{multline}
\end{proposition}

\begin{proof}
Since the sum in question is empty when $3\le y<9$, we may assume that $y\ge9$.
We begin by writing
\begin{align}
\sum_{\substack{p\le y/3\\ p\ne w_1,w_2,\ldots, w_n \\ \nu_q(p-1)=\alpha}} & \frac{p+O(1)}{p^2} \bigg(\frac{(\log\log(y/p))^k}{(\log(y/p))^\gamma} + O\bigg( \frac {(\log\log y)^{k-1}}{(\log(y/p))^{\min\{\gamma,1\}}}\bigg) \bigg) \notag \\
&= \bigg\{ \sum_{\substack{p\le y/3\\ \nu_q(p-1)=\alpha}}\frac{(\log\log(y/p))^k}{p(\log(y/p))^\gamma}+O\bigg( \sum_{\substack{1\le i\le n \\ w_i \le y/3}} \frac{(\log\log(y/w_i))^k}{w_i(\log(y/w_i))^\gamma} \bigg) \bigg\} \notag \\
&\qquad{}+ O\bigg(\sum_{\substack{p\le y/3 \\ \nu_q(p-1)=\alpha}}\frac{(\log\log(y/p))^k}{p^2(\log(y/p))^\gamma} + \sum_{\substack{p\le y/3 \\ \nu_q(p-1)=\alpha}}\frac {(\log\log y)^{k-1}}{p(\log(y/p))^{\min\{\gamma,1\}}}\bigg). \notag \\
&= \frac{(\log\log y)^{k+1}}{q^\alpha\log^\gamma{y}}+O_{k,\gamma}\bigg( \frac{(\log\log y)^k}{(\log y)^{\min\{\gamma,1\}}}\bigg) +O\bigg( \sum_{\substack{1\le i\le n \\ w_i \le y/3}} \frac{(\log\log y)^k}{w_i(\log(y/w_i))^\gamma} \bigg) \label{simplify Os} \\
&\qquad{}+ O\bigg(\sum_{\substack{p\le y/3 \\ \nu_q(p-1)=\alpha}}\frac{(\log\log y)^k}{p^2(\log(y/p))^\gamma} + \sum_{\substack{p\le y/3 \\ \nu_q(p-1)=\alpha}}\frac {(\log\log y)^{k-1}}{p(\log(y/p))^{\min\{\gamma,1\}}}\bigg) \notag
\end{align}
by Lemma~\ref{recsum lemma}.
Since $w_i \log(y/w_i) \gg \log y$ for $y\ge9$ and $w_i\le y/3$, the first error term sum on the right-hand side is $\ll_n (\log \log y)^k (\log y)^{-\gamma}$, while
\begin{align*}
\sum_{\substack{p\le y/3 \\ \nu_q(p-1)=\alpha}}\frac{(\log\log y)^k}{p^2(\log(y/p))^\gamma} &\ll_\gamma \frac{(\log\log y)^k}{(\log y)^\gamma} \\
\sum_{\substack{p\le y/3 \\ \nu_q(p-1)=\alpha}} \frac {(\log\log y)^{k-1}}{p(\log(y/p))^{\min\{\gamma,1\}}} &\ll_\gamma (\log\log y)^{k-1}\frac{\log\log y}{(\log y)^{\min\{\gamma,1\}}}
\end{align*}
by Lemma~\ref{sumsq} and Corollary~\ref{sumcor}. Therefore the error terms in equation~\eqref{simplify Os} are all majorized by the error term in equation~\eqref{recsum eq}.
\end{proof}

\begin{definition}\label{S_q}
Let $k$ be a nonnegative real number, and let~$q$ be prime. Define
\begin{equation} \label{Sqxk def}
S_q(x;k)=\frac{(\log\log x)^k}{(\log x)^{1/(q-1)}}+O_q\bigg( \frac{(\log\log x)^{k-1}}{(\log x)^{1/(q-1)}}\bigg).
\end{equation}
Further, for $\alpha_1,\ldots,\alpha_i\in \N$, define
\begin{multline*}
S_q(x;k;\alpha_1,\ldots,\alpha_i)\\
=\sum_{\substack{p_1\le x/3\\\nu_q(p_1-1)=\alpha_1}} \frac{p_1+O(1)}{p_1^2} \sum_{\substack{p_2\le x/3p_1\\p_2\ne p_1\\\nu_q(p_2-1)=\alpha_2}} \frac{p_2+O(1)}{p_2^2} \cdots\sum_{\substack{p_{i}\le x/3p_1\cdots p_{i-1}\\p_{i}\ne p_1,\ldots,p_{i-1}\\\nu_q(p_i-1)=\alpha_i}} \bigg\{ \frac{p_i+O(1)}{p_i^2} \\
\times \bigg( \frac{(\log\log(x/p_1\cdots p_{i}))^k}{(\log(x/p_1\cdots p_{i}))^{1/(q-1)}}+O_q\bigg( \frac {(\log\log(x/p_1\cdots p_i))^{k-1}}{(\log(x/p_1\cdots p_{i}))^{1/(q-1)}}\bigg) \bigg) \bigg\}.
\end{multline*}
\end{definition}

Note that the expressions $S_q(x;k)$ and $S_q(x;k;\alpha_1,\ldots,\alpha_i)$ are given by asymptotic, not explicit, formulas. For instance, when $i=1$, applying Proposition \ref{recsum} yields
\begin{align*}
S_q(x;k;\alpha_1)&=\sum_{\substack{p_1\le x/3\\\nu_q(p_1-1)=\alpha_1}} \frac{p_1+O(1)}{p_1^2} \bigg( \frac{(\log\log(x/p_1))^k}{(\log(x/p_1))^{1/(q-1)}}+O_q\bigg( \frac {(\log\log(x/p_1))^{k-1}}{(\log(x/p_1))^{1/(q-1)}}\bigg) \bigg)\\
&=\frac{(\log\log x)^{k+1}}{q^{\alpha_1}(\log x)^{1/(q-1)}}+O_q\bigg( \frac{(\log\log x)^k}{(\log x)^{1/(q-1)}}\bigg),
\end{align*}
which, by comparison to equation~\eqref{Sqxk def}, is the same as the expression $S_q(x;k+1)/q^{\alpha_1}$.
Here, we are not claiming that $S_q(x;k;\alpha_1)$ must be exactly equal to $S_q(x;k+1)/q^{\alpha_1}$, but rather that these two expressions have identical main terms and error terms of equal magnitude; in particular, we may freely replace $S_q(x;k;\alpha_1)$ by $S_q(x;k+1)/q^{\alpha_1}$ in any expression.

This observation generalizes to any natural number $i$, resulting in the following proposition. 

\begin{proposition} \label{recarg}
Let $k$ be a nonnegative real number, let~$q$ be prime, and let $\alpha_1,\ldots,\alpha_j\in\N$.
For any $1\le i\le j$, the expressions
\[
S_q(x;k;\alpha_1,\ldots,\alpha_i) \quad\text{and}\quad {q^{-\alpha_i}}S_q(x;k+1;\alpha_1,\ldots,\alpha_{i-1})
\]
have the same main terms and error terms of equal magnitude, so that we may freely replace $S_q(x;k;\alpha_1,\ldots,\alpha_i)$ with ${q^{-\alpha_i}} S_q(x;k+1;\alpha_1,\ldots,\alpha_{i-1})$ in any expression.
In particular, the expressions
\[
S_q(x;k;\alpha_1,\ldots,\alpha_j) \quad\text{and}\quad {q^{-\sum_{i=1}^j \alpha_i}} S_q(x;k+j)
\]
have the same main terms and error terms of equal magnitude, so that we may freely replace $S_q(x;k;\alpha_1,\ldots,\alpha_j)$ with ${q^{-\sum_{i=1}^j \alpha_i}}S_q(x;k+j)$ in any expression (as long as we note that the error term in equation~\eqref{Sqxk def} will depend on~$j$ as well as~$q$).
\end{proposition}

\begin{proof}
Applying Proposition~\ref{recsum}, with $y=x/p_1\dots p_{i-1}$, to the innermost sum in Definition~\ref{S_q}, we see that
\begin{align*}
S_q(x;k;&\alpha_1,\ldots,\alpha_i)\\
&=\sum_{\substack{p_1\le x/3\\\nu_q(p_1-1)=\alpha_1}} \frac{p_1+O(1)}{p_1^2} \cdots\sum_{\substack{p_{i-1}\le x/3p_1\cdots p_{i-2}\\p_{i-1}\ne p_1,\ldots,p_{i-2}\\\nu_q(p_{i-1}-1)=\alpha_{i-1}}} \bigg\{ \frac{p_1+O(1)}{p_1^2} \\
&\qquad{}\times \bigg( \frac{(\log\log(x/p_1\cdots p_{i-1}))^{k+1}}{q^{\alpha_i}(\log(x/p_1\cdots p_{i-1}))^{1/(q-1)}}+O_q\bigg( \frac {(\log\log(x/p_1\cdots p_{i-1}))^{k}}{(\log(x/p_1\cdots p_{i-1}))^{1/(q-1)}}\bigg) \bigg) \bigg\}.
\end{align*}
The last assertion follows from a trivial induction.
\end{proof}

\begin{definition}\label{ep_q}
Let $\gamma$ be a positive real number, and let~$q$ be prime. Define $\ep_q(x,\gamma) = (\log x)^{-\gamma}$.
Further, for any $\alpha_1,\ldots,\alpha_i\in\N$, define
\[
\ep_q(x,\gamma;\alpha_1,\ldots,\alpha_i)=\sum_{\substack{p_1\le x/3\\\nu_q(p_1-1)=\alpha_1}}\frac 1{p_1}\sum_{\substack{p_2\le x/3p_1\\p_2\ne p_1\\\nu_q(p_2-1)=\alpha_2}}\frac 1{p_2}\cdots \sum_{\substack{p_i\le x/3p_1\cdots p_{i-1}\\p_i\ne p_1,\ldots,p_{i-1}\\\nu_q(p_i-1)=\alpha_i}}\frac 1{p_i}\bigg( \log\frac x{p_1 \cdots p_i}\bigg)^{-\gamma}.
\]
\end{definition}

\begin{proposition} \label{recarg2}
Let $\gamma$ be a positive real number, let~$q$ be prime, and let $\alpha_1,\ldots,\alpha_j\in\N$. For any $1\le i\le j$ and for any $x\ge3$,
\[
\ep_q(x,\gamma;\alpha_1,\ldots,\alpha_i)\ll_q \ep_q(x,\gamma;\alpha_1,\ldots,\alpha_{i-1}) \log\log x.
\]
In particular, $\ep_q(x,\gamma;\alpha_1,\ldots,\alpha_j) \ll_{q,j} (\log\log x)^j/(\log x)^\gamma$.
\end{proposition}

\begin{proof}
If $x/p_1\dots p_{i-1} < 3$ then the innermost sum in the definition of $\ep_q(x,\gamma;\alpha_1,\ldots,\alpha_i)$ is empty; otherwise, applying Lemma~\ref{sum} (with $y=x/p_1\dots p_{i-1}$) to the innermost sum, we obtain
\begin{align*}
\ep_q(x,\gamma; &\, \alpha_1,\ldots,\alpha_i)\\
&\ll \sum_{\substack{p_1\le x/3\\\nu_q(p_1-1)=\alpha_1}}\frac 1{p_1}\sum_{\substack{p_2\le x/3p_1\\p_2\ne p_1\\\nu_q(p_2-1)=\alpha_2}}\frac 1{p_2}\cdots \sum_{\substack{p_{i-1}\le x/3p_1\cdots p_{i-2}\\p_{i-1}\ne p_1,\ldots,p_{i-2}\\\nu_q(p_{i-1}-1)=\alpha_{i-1}}}\frac 1{p_{i-1}}\bigg( \frac{\log\log(x/p_1\cdots p_{i-1})}{q^{\alpha_i}(\log(x/p_1\cdots p_{i-1}))^\gamma}\bigg) \\
&\ll \sum_{\substack{p_1\le x/3\\\nu_q(p_1-1)=\alpha_1}}\frac 1{p_1}\sum_{\substack{p_2\le x/3p_1\\p_2\ne p_1\\\nu_q(p_2-1)=\alpha_2}}\frac 1{p_2}\cdots \sum_{\substack{p_{i-1}\le x/3p_1\cdots p_{i-2}\\p_{i-1}\ne p_1,\ldots,p_{i-2}\\\nu_q(p_{i-1}-1)=\alpha_{i-1}}}\frac 1{p_{i-1}}\bigg( \frac{\log\log x}{(\log(x/p_1\cdots p_{i-1}))^\gamma}\bigg) \\
&= \ep_q(x,\gamma;\alpha_1,\ldots,\alpha_{i-1}) \log\log x.
\end{align*}
(We check that this calculation is valid even in the case $i=1$, where the above notation is obfuscatory.)
The last assertion follows from the definition of $\ep_q(x,\gamma)$ and a trivial induction.
\end{proof}

\subsection{Evaluation of $D_0(H,x)$}

We now have all the tools we need to evaluate the counting function $D_0(H,x)$ from Definition~\ref{DkHx def}, which is the majority of the work needed to establish Theorem~\ref{main}.
Since we can apply Proposition~\ref{sdsum2} only when $x/p_1\cdots p_j \ge 3$, we start by splitting the sum in equation~\eqref{D_0}, so that Lemma~\ref{initial D0 prop} becomes
\begin{multline} \label{D_0later}
D_0(\Z_{q^\balpha},x)= C(\balpha) \bigg( \sum_{\substack{p_1\le x/3\\ \nu_q(p_1-1)=\alpha_1}}\sum_{\substack{p_2\le x/3p_1\\p_2\ne p_1\\\nu_q(p_2-1)=\alpha_2}}\cdots \sum_{\substack{p_j\le x/3p_1\cdots p_{j-1}\\p_j\ne p_1,\ldots,p_{j-1}\\\nu_q(p_j-1)=\alpha_j}}\sum_{\substack{m\le x/p_1\cdots p_j\\q\dnd m\\ (t\mid m\text{ and }t\equiv1\mod q \Rightarrow t\in\{p_1,\dots,p_j\}}}1 +{} \\
\sum_{\substack{p_1\le x/3 \\ \nu_q(p_1-1)=\alpha_1}}\sum_{\substack{p_2\le x/3p_1\\p_2\ne p_1\\\nu_q(p_2-1)=\alpha_2}}\cdots \sum_{\substack{p_{j-1}\le x/3p_1\cdots p_{j-2} \\ p_{j-1}\ne p_1,\ldots,p_{j-2} \\ \nu_q(p_{j-1}-1)=\alpha_{j-1}}} \sum_{\substack{x/3p_1\cdots p_{j-1}<p_j\le x/p_1\cdots p_{j-1}\\p_j\ne p_1,\ldots,p_{j-1}\\\nu_q(p_j-1)=\alpha_j}} \hskip-7mm \sum_{\substack{m\le x/p_1\cdots p_j\\q\dnd m\\ (t\mid m\text{ and }t\equiv1\mod q \Rightarrow t\in\{p_1,\dots,p_j\}}} \hskip-2mm 1 \bigg).
\end{multline}
In the next two propositions we estimate the second sum in equation~\eqref{D_0later} and then asymptotically evaluate the first sum.

\begin{lemma}\label{secondsum}
Let $q$ be an odd prime, and let $\balpha=(\alpha_1,\dots,\alpha_j)$ be a partition. For $x\ge3$,
\begin{multline*}
\sum_{\substack{p_1\le x/3 \\ \nu_q(p_1-1)=\alpha_1}}\sum_{\substack{p_2\le x/3p_1\\p_2\ne p_1\\\nu_q(p_2-1)=\alpha_2}}\cdots \sum_{\substack{p_{j-1}\le x/3p_1\cdots p_{j-2} \\ p_{j-1}\ne p_1,\ldots,p_{j-2} \\ \nu_q(p_{j-1}-1)=\alpha_{j-1}}} \sum_{\substack{x/3p_1\cdots p_{j-1}<p_j\le x/p_1\cdots p_{j-1}\\p_j\ne p_1,\ldots,p_{j-1}\\\nu_q(p_j-1)=\alpha_j}} \hskip-7mm \sum_{\substack{m\le x/p_1\cdots p_j\\q\dnd m\\ (t\mid m\text{ and }t\equiv1\mod q \Rightarrow t\in\{p_1,\dots,p_j\}}} \hskip-2mm 1 \\
\ll_{q,j} \frac{x(\log\log x)^{j-1}}{\log x}.
\end{multline*}
\end{lemma}

\begin{proof}
Since $x/p_1p_2\cdots p_j<3$, the innermost sum has at most two terms, and thus
\begin{multline*}
\sum_{\substack{p_1\le x/3 \\ \nu_q(p_1-1)=\alpha_1}}\sum_{\substack{p_2\le x/3p_1\\p_2\ne p_1\\\nu_q(p_2-1)=\alpha_2}}\cdots \sum_{\substack{p_{j-1}\le x/3p_1\cdots p_{j-2} \\ p_{j-1}\ne p_1,\ldots,p_{j-2} \\ \nu_q(p_{j-1}-1)=\alpha_{j-1}}} \sum_{\substack{x/3p_1\cdots p_{j-1}<p_j\le x/p_1\cdots p_{j-1}\\p_j\ne p_1,\ldots,p_{j-1}\\\nu_q(p_j-1)=\alpha_j}} \hskip-7mm \sum_{\substack{m\le x/p_1\cdots p_j\\q\dnd m\\ (t\mid m\text{ and }t\equiv1\mod q \Rightarrow t\in\{p_1,\dots,p_j\}}} \hskip-2mm 1 \\
\ll \sum_{\substack{p_1\le x/3\\\nu_q(p_1-1)=\alpha_1}}\sum_{\substack{p_2\le x/3p_1\\p_2\ne p_1\\\nu_q(p_2-1)=\alpha_2}}\cdots \sum_{\substack{p_{j-1}\le x/3p_1\cdots p_{j-2} \\ p_{j-1}\ne p_1,\ldots,p_{j-2} \\ \nu_q(p_{j-1}-1)=\alpha_{j-1}}} \sum_{\substack{x/3p_1\cdots p_{j-1}<p_j\le x/p_1\cdots p_{j-1}\\p_j\ne p_1,\ldots,p_{j-1}\\\nu_q(p_j-1)=\alpha_j}} 1.
\end{multline*}
If $x/p_1\cdots p_{j-1} < 3$ then the innermost sum vanishes; otherwise Chebyshev's estimate gives
\begin{align*}
\sum_{\substack{p_1\le x/3\\\nu_q(p_1-1)=\alpha_1}} & \sum_{\substack{p_2\le x/3p_1\\p_2\ne p_1\\\nu_q(p_2-1)=\alpha_2}}\cdots \sum_{\substack{p_{j-1}\le x/3p_1\cdots p_{j-2}\\p_{j-1}\ne p_1,\ldots,p_{j-2}\\\nu_q(p_{j-1}-1)=\alpha_{j-1}}} \sum_{\substack{x/3p_1\cdots p_{j-1}<p_j\le x/p_1\cdots p_{j-1}\\p_j\ne p_1,\ldots,p_{j-1}\\\nu_q(p_j-1)=\alpha_j}}1\\
&\le \sum_{\substack{p_1\le x/3\\\nu_q(p_1-1)=\alpha_1}}\sum_{\substack{p_2\le x/3p_1\\p_2\ne p_1\\\nu_q(p_2-1)=\alpha_2}}\cdots \sum_{\substack{p_{j-1}\le x/3p_1\cdots p_{j-2}\\p_{j-1}\ne p_1,\ldots,p_{j-2}\\\nu_q(p_{j-1}-1)=\alpha_{j-1}}} \pi(x/p_1p_2\cdots p_{j-1})\\
&\ll \sum_{\substack{p_1\le x/3\\\nu_q(p_1-1)=\alpha_1}}\sum_{\substack{p_2\le x/3p_1\\p_2\ne p_1\\\nu_q(p_2-1)=\alpha_2}}\cdots \sum_{\substack{p_{j-1}\le x/3p_1\cdots p_{j-2}\\p_{j-1}\ne p_1,\ldots,p_{j-2}\\\nu_q(p_{j-1}-1)=\alpha_{j-1}}}\frac{x/p_1\cdots p_{j-1}}{\log(x/p_1\cdots p_{j-1})} \\
&=x \ep_q(x,1;\alpha_1,\ldots,\alpha_{j-1})
\end{align*}
in the notation of Definition~\ref{ep_q}. The lemma now follows directly from Proposition~\ref{recarg2}.
\end{proof}

\begin{proposition}\label{firstsum}
Let $q$ be an odd prime, and let $\alpha_1,\dots,\alpha_j\in\N$. For any $x\ge3$,
\begin{multline} \label{J def}
\sum_{\substack{p_1\le x/3\\ \nu_q(p_1-1)=\alpha_1}}\sum_{\substack{p_2\le x/3p_1\\p_2\ne p_1\\\nu_q(p_2-1)=\alpha_2}}\cdots \sum_{\substack{p_j\le x/3p_1\cdots p_{j-1}\\p_j\ne p_1,\ldots,p_{j-1}\\\nu_q(p_j-1)=\alpha_j}}\sum_{\substack{m\le x/p_1\cdots p_j\\q\dnd m\\ (t\mid m\text{ and }t\equiv1\mod q \Rightarrow t\in\{p_1,\dots,p_j\}}}1\\
= \frac{B_q}{q^{\sum_{i=1}^j\alpha_i}} \frac{x(\log\log x)^j}{(\log x)^{1/(q-1)}}+O_q\bigg( \frac{x(\log\log x)^{j-1}}{(\log x)^{1/(q-1)}}\bigg),
\end{multline}
where $B_q$ is as in Definition~\ref{Bq def}.
\end{proposition}

\begin{proof}
Throughout this proof, let~$J$ denote the left-hand side of equation~\eqref{J def}.
Note that the condition $p_j\le x/3p_1\cdots p_{j-1}$ in the second-to-last sum implies that the bound $x/p_1\cdots p_j$ in the innermost sum is at least~$3$. Therefore, by Proposition~\ref{sdsum2},
\begin{align*}
J&=\sum_{\substack{p_1\le x/3\\\nu_q(p_1-1)=\alpha_1}}\sum_{\substack{p_2\le x/3p_1\\p_2\ne p_1\\\nu_q(p_2-1)=\alpha_2}}\cdots \sum_{\substack{p_j\le x/3p_1\cdots p_{j-1}\\p_j\ne p_1,\ldots,p_{j-1}\\\nu_q(p_j-1)=\alpha_j}}\Bigg\{ B_q\frac x{p_1\cdots p_j} \bigg( \log\frac x{p_1\cdots p_j}\bigg)^{-1/(q-1)} \prod_{i=1}^j\bigg(1-\frac 1{p_i}\bigg)^{-1} \\
&\qquad{}+O\bigg( \frac x{p_1\cdots p_j}\bigg( \log\frac x{p_1 \cdots p_j}\bigg)^{-1-1/(q-1)}\bigg) \Bigg\} \\
&=B_q x\sum_{\substack{p_1\le x/3\\\nu_q(p_1-1)=\alpha_1}}\cdots \sum_{\substack{p_j\le x/3p_1\cdots p_{j-1}\\p_j\ne p_1,\ldots,p_{j-1}\\\nu_q(p_j-1)=\alpha_j}}\frac 1{p_1\cdots p_j} \bigg( \log\frac x{p_1\cdots p_j}\bigg)^{-1/(q-1)}\prod_{i=1}^j\bigg(1-\frac 1{p_i}\bigg)^{-1}\\
&\hspace{1.5cm}+O\bigg(x\sum_{\substack{p_1\le x/3\\\nu_q(p_1-1)=\alpha_1}}\sum_{\substack{p_2\le x/3p_1\\p_2\ne p_1\\\nu_q(p_2-1)=\alpha_2}}\cdots \sum_{\substack{p_j\le x/3p_1\cdots p_{j-1}\\p_j\ne p_1,\ldots,p_{j-1}\\\nu_q(p_j-1)=\alpha_j}}\frac 1{p_1\cdots p_j}\bigg( \log\frac x{p_1 \cdots p_j}\bigg)^{-1-1/(q-1)}\bigg) \\
&=B_qx S_q(x,0;\alpha_1,\dots,\alpha_j) + O\bigg( x \ep_q\bigg( x,1+\frac1{q-1};\alpha_1,\dots,\alpha_j \bigg) \bigg)
\end{align*}
in the notation of Definitions~\ref{S_q} and~\ref{ep_q}. Thus, by Propositions~\ref{recarg} and~\ref{recarg2},
\begin{align*}
J &= B_qx \frac1{q^{\sum_{i=1}^j} \alpha_i} S_q(x,j) + O\bigg( \frac{x (\log\log x)^j}{(\log x)^{1+1/(q-1)}} \bigg) \\
&= B_qx \frac1{q^{\sum_{i=1}^j} \alpha_i} \bigg( \frac{(\log\log x)^j}{(\log x)^{1/(q-1)}}+O_q\bigg( \frac{(\log\log x)^{j-1}}{(\log x)^{1/(q-1)}}\bigg) \bigg) + O\bigg( \frac{x (\log\log x)^j}{(\log x)^{1+1/(q-1)}} \bigg)
\end{align*}
by equation~\eqref{Sqxk def}, which establishes the proposition.
\end{proof}

The work in this section leads immediately to an asymptotic formula for~$D_0(\Z_{q^\balpha},x)$.

\begin{theorem} \label{D0 prop}
Let $q$ be an odd prime, and let $\balpha=(\alpha_1,\dots,\alpha_j)$ be a partition. For any $x\ge3$,
\[
D_0(\Z_{q^\balpha},x) = C(\balpha) \frac{B_q}{q^{\sum_{i=1}^j\alpha_i}} \frac{x(\log\log x)^j}{(\log x)^{1/(q-1)}} + O_{q,\balpha}\bigg(\frac{x(\log\log x)^{j-1}}{(\log x)^{1/(q-1)}}\bigg).
\]
\end{theorem}

\begin{proof}
Thanks to the expression~\eqref{D_0later}, the theorem follows immediately from Proposition~\ref{firstsum} and Lemma~\ref{secondsum}.
\end{proof}

\subsection{Proof of Theorem~\ref{main}}

As mentioned earlier, each of the counting functions $D_k(H,x)$, if nonzero, is equal to $D_0(\Z_{H'},x)$ for a certain $q$-group~$H'$ depending on~$H$ and~$k$. We make this precise enough for our purposes in the following lemma.

\begin{lemma} \label{handle Dk lemma}
Let $q$ be an odd prime, let $\balpha=(\alpha_1,\dots,\alpha_j)$ be a partition, and let $k\in\N$.
\begin{enumerate}
\item When $k=1$, we have $D_1(\Z_{q^\balpha},x) = D_0(\Z_{q^\balpha},\frac xq)$.
\item If $k\ge \alpha_1+2$, then $D_k(\Z_{q^\balpha},x)=0$.
\item For any $k\ge2$, we have $D_k(\Z_{q^\balpha},x)\ll_{q,\balpha,k} x(\log\log x)^{j-1}/(\log x)^{1/(q-1)}$.
\end{enumerate}
\end{lemma}

\begin{proof}
In all parts, the integers counted by $D_k(\Z_{q^\balpha},x) = \#\{n\le x \colon \nu_q(n)=k,\, G_q(n)=\Z_{q^\balpha}\}$ can be written as $n=q^km$ where $m\le \frac x{q^k}$ and $q\nmid m$; we also have $\Z_n^\times \cong \Z_{q^{k-1}(q-1)} \times \Z_m^\times$, and in particular the Sylow $q$-subgroup of $\Z_n^\times$ is congruent to the product of $\Z_{q^{k-1}}$ and the Sylow $q$-subgroup of $\Z_m^\times$.

When $k=1$, these two Sylow $q$-subgroups are identical, and therefore $D_1(\Z_{q^\balpha},x) = D_0(\Z_{q^\balpha},\frac xq)$ as claimed in part~(a).

When $k\ge\alpha_1+2$, the Sylow $q$-subgroup of $\Z_n^\times$ will include a copy of $\Z_{q^{k-1}}$; since the largest primary subgroup of $\Z_{q^\balpha}$ is $\Z_{q^{\alpha_1}}$, the fact that $k-1>\alpha_1$ makes it impossible for that Sylow $q$-subgroup to equal $\Z_{q^\balpha}$, as claimed in part~(b). The same is true if $k-1\le\alpha_1$ but $k-1\notin\balpha$.

Finally, suppose $k\ge2$ and $k-1\in\balpha$. Let $\balpha'=\balpha\setminus\{k-1\}$ denote the partition obtained from~$\balpha$ by removing one occurrence of $k-1$. Then the fact that $\Z_n^\times \cong \Z_{q^{k-1}(q-1)} \times \Z_m^\times$ implies that the Sylow $q$-subgroup of $\Z_m^\times$ equals $\Z_{q^{\balpha'}}$, and therefore
\[
D_k(\Z_{q^\balpha},x) = D_0\bigg( \Z_{q^{\balpha'}}, \frac x{q^k} \bigg) \ll_{q,\balpha'} \frac{(x/q^k)(\log\log (x/q^k))^{j-1}}{(\log (x/q^k))^{1/(q-1)}} \ll_{q,\balpha,k} \frac{x(\log\log x)^{j-1}}{(\log x)^{1/(q-1)}},
\]
as claimed in part~(c).
\end{proof}

We have now completed the last preparatory step necessary to prove our main theorem, which we do after defining the leading constant that emerges from the calculation.

\begin{definition} \label{Eqalpha def}
For a prime $q$ and a partition $\balpha=(\alpha_1,\dots,\alpha_j)$, define
\[
E_q(\balpha)=\frac{q+1}{q^{1+\sum_{i=1}^j \alpha_i}}.
\]
Recall also~$B_q$ and~$C(\balpha)$ from Definition~\ref{Bq def} and equation~\eqref{Calpha def}, respectively:
\begin{align*}
B_q &= \frac{1}{\Gamma(1-1/(q-1))}\bigg(1-\frac1q \bigg)^{1-1/(q-1)}\prod_{\substack{p\ne q\\ p\not\equiv 1 \mod q}} \bigg(1-\frac1{p^{k_p}} \bigg)^{-1/k_p}\prod_{\chi\ne \chi_0}L(1,\chi)^{-1/(q-1)} \\
C(\balpha) &= \prod_{u=1}^\infty \frac1{(a_u-a_{u+1})!} = \prod_{u=1}^{\alpha_1} \frac1{(a_u-a_{u+1})!},
\end{align*}
where $(a_1,a_2,\dots)$ is the conjugate of the partition~$\balpha$.
Then, given the finite abelian $q$-group $\Z_{q^\balpha}$, define the constant
\[
K(\Z_{q^\balpha})=B_qC(\balpha)E_q(\balpha).
\]
\end{definition}

\begin{proof}[Proof of Theorem~\ref{main}]
Write $\balpha=(\alpha_1,\dots,\alpha_j)$.
Beginning with Definition~\ref{DkHx def},
\begin{align*}
D(\Z_{q^\balpha},x) = \sum_{k=0}^{\infty} D_k(\Z_{q^\balpha},x) &= \sum_{k=0}^{\alpha_1+1} D_k(\Z_{q^\balpha},x) \\
&= D_0(\Z_{q^\balpha},x) + D_0\big( \Z_{q^\balpha},\tfrac xq \big) + O_{q,\balpha} \bigg( \frac{x(\log\log x)^{j-1}}{(\log x)^{1/(q-1)}} \bigg)
\end{align*}
by Lemma~\ref{handle Dk lemma}. Now Theorem~\ref{D0 prop} gives
\begin{multline*}
D(\Z_{q^\balpha},x) = C(\balpha) \frac{B_q}{q^{\sum_{i=1}^j\alpha_i}} \frac{x(\log\log x)^j}{(\log x)^{1/(q-1)}} + C(\balpha) \frac{B_q}{q^{\sum_{i=1}^j\alpha_i}} \frac{(x/q)(\log\log (x/q))^j}{(\log (x/q))^{1/(q-1)}} \\
+ O_{q,\balpha} \bigg( \frac{x(\log\log x)^{j-1}}{(\log x)^{1/(q-1)}} \bigg),
\end{multline*}
which establishes the theorem since $(1+1/q)/q^{\sum_{i=1}^j\alpha_i} = E_q(\balpha)$ and $j=\ell(\balpha)$.
\end{proof}

\section{Maximally non-cyclic multiplicative groups---proof of Theorem~\ref{mnc theorem}} \label{mnc section}

Rather than focusing on local Sylow subgroups, we now wish to focus on the global structure of the group~$\Z_n^\times$, and in particular (as described in the introduction) when this group is as far from being cyclic as possible. To define this notion precisely, recall that the {\em primary decomposition} of a finite abelian group~$G$ is the unique isomorphism of the shape $G\cong\Z_{p_1^{r_1}}\times\cdots\times\Z_{p_k^{r_k}}$ where the $p_j^{r_j}$ are prime powers (with $r_j=1$ possible), while its {\em invariant factor decomposition} is the unique isomorphism of the shape $G\cong\Z_{d_1}\times\cdots\times \Z_{d_\ell}$ where $d_1\mid d_2\mid\cdots\mid d_\ell$.

\begin{definition} \label {mncdef}
Let $G$ be a finite abelian group of cardinality~$m$. We say that~$G$ is \textit{maximally non-cyclic} if any of the four following equivalent conditions hold:
\begin{enumerate}
\item each factor of the primary decomposition of~$G$ is of the form $\Z_p$ for some prime~$p$;
\item for any prime $p$, the Sylow $p$-subgroup of~$G$ is an elementary $p$-group, that is, is of the form $\Z_p\times \Z_p\times\cdots \times\Z_p$;
\item the invariant factors~$d_j$ are squarefree for every $1\le j\le \ell$;
\item the largest invariant factor~$d_\ell$ is minimal among all finite abelian groups of order~$m$.
\end{enumerate}
That these four conditions are indeed equivalent is a straightforward exercise in undergraduate algebra. (We remark in passing that it is possible for a finite abelian group to be both cyclic and maximally non-cyclic: such groups are precisely the cyclic groups of squarefree order, which are the orders for which there exists exactly one finite abelian group.)
\end{definition}

The most intuitive definitions of maximally non-cyclic are conditions~(a)/(b) (which are nearly identical) and condition~(d); condition~(c), on the other hand, is less intuitive but turns out to be useful in the proof of the characterization of integers whose multiplicative group is maximally non-cyclic.
We remark that these equivalent conditions imply that the length~$\ell$ of the invariant factor decomposition is maximal among all finite abelian groups of order~$m$, although this is not a two-way implication as shown by the examples $\Z_6\times \Z_6$ (which is maximally non-cyclic) and $\Z_2\times \Z_{18}$ (which is not) corresponding to $m=36$ and $\ell=2$.

Our goal in this section is to asymptotically evaluate the counting function for the number of integers~$n\le x$ such that $\Z_n^\times$ is maximally non-cyclic. As in Section~\ref{DHx section}, we can accomplish this evaluation by giving a characterization of this group-theoretic property in terms of the prime factorization of~$n$.

\begin{proposition} \label{max non cyc characterization prop}
For any $n\in \N$, the group $\Z_n^\times$ is maximally non-cyclic if and only if:
\begin{enumerate}
\item $2^4\dnd n$;
\item $p^3\dnd n$ for every odd prime~$p$; and
\item $p-1$ is squarefree for every $p\mid n$.
\end{enumerate}
\end{proposition}

\begin{proof}
By part~(a) of Definition~\ref{mncdef}, together with uniqueness of primary decompositions, a product $H_1\times\cdots\times H_k$ of finite abelian groups is maximally non-cyclic if and only if each~$H_j$ is maximally non-cyclic. In particular, $\Z_n^\times$ is maximally non-cyclic if and only if $\Z_{p^r}^\times$ is maximally non-cyclic for every $p^r\| n$. Since
\begin{equation*}
\Z_{2^r}^\times \cong \begin{cases}
\text{trivial},  &\text{ if } r\le1,\\
\Z_2,  &\text{ if } r=2\\
\Z_2\times \Z_2,  &\text{ if } r=3\\
\Z_{2^{r-2}}\times \Z_2,  &\text{ if } r\ge 4,
\end{cases}
\end{equation*}
we see that $2^4\dnd n$ is one necessary condition. Similarly, since for odd primes~$p$,
\begin{equation*}
\Z_{\phi(p^r)} \cong \begin{cases}
\Z_{p-1}, &\text{ if } r=1\\
\Z_p\times \Z_{p-1},  &\text{ if } r=2\\
\Z_{p^{r-1}}\times\Z_{p-1},  &\text{ if } r\ge 3,
\end{cases}
\end{equation*}
we see that $p^3\dnd n$ for all odd primes~$p$ is a second necessary condition. Finally, by part~(c) of Definition~\ref{mncdef}, the cyclic group $\Z_{p-1}$ is maximally non-cyclic if and only if $p-1$ is squarefree, which is the source of the third and final condition.
\end{proof}

\begin{definition}\label{Artin}
Let $\xi$ denote Artin's constant
\[
\xi=\prod_p \bigg(1-\frac 1{p(p-1)}\bigg),
\]
and define the positive constant
\begin{align*}
A &= \frac {15}{14\Gamma(\xi)}\lim_{x\to\infty}\bigg\{ \prod_{\substack{p\le x\\\mu^2(p-1)=1}}\bigg(1+\frac 1p+\frac 1{p^2}\bigg) \prod_{p\le x}\bigg(1-\frac 1p\bigg)^\xi\bigg\} \\
&= \frac {15}{14\Gamma(\xi)} \prod_p \bigg(1+\frac{(p+1)\mu^2(p-1)}{p^2}\bigg) \bigg(1-\frac 1p\bigg)^\xi.
\end{align*}
We will see that the product defining~$A$ converges in the proof of Theorem~\ref{mnc theorem} below.
\end{definition}

It is known that Artin's constant is also the density of the primes~$p$ with the property that $p-1$ is squarefree; we provide a proof for the sake of completeness.

\begin{lemma}\label{sfsum}
For $x\ge2$, we have
$\displaystyle
\#\{p\le x\colon p-1 \text{ is squarefree}\} = \frac{\xi x}{\log x}  + O\bigg( \frac {x}{\log^2 x}\bigg).
$
\end{lemma}

\begin{proof}
Using the well-known identity $\mu^2(n)=\sum_{d^2\mid n}\mu(d)$~\cite[equation~(2.4)]{MV}, we have
\begin{align*}
\#\{p\le x\colon p-1 \text{ is squarefree}\} &= \sum_{p\le x}\mu(p-1)^2 \\
&= \sum_{p\le x}\sum_{d^2\mid p-1}\mu(d) = \sum_{d^2\le x}\mu(d)\sum_{\substack{p\le x\\d^2\mid p-1}}1 = \sum_{d\le \sqrt{x}}\mu(d)\pi(x;d^2,1).
\end{align*}
For $d\le \log^2 x$ we use the Siegel--Walfisz theorem~\cite[Corollary~11.21]{MV}
\[
\pi(x;d^2,1)=\frac{\li(x)}{\phi(d^2)}+O(xe^{-c_1\sqrt{\log x}}) = \frac{\li(x)}{d\phi(d)} + O\bigg( \frac x{\log^4x} \bigg),
\]
while for $\log^2x < d \le \sqrt x$ we use the trivial estimate $\pi(x;d^2,1)\le x/d^2$. We find that
\begin{align*}
\sum_{d\le \sqrt{x}}\mu(d)\pi(x;d^2,1) &= \sum_{d\le\log^2x} \mu(d) \bigg( \frac{\li(x)}{d\phi(d)} + O\bigg( \frac x{\log^4x} \bigg) \bigg) + O\bigg( \sum_{\log^2x<d\le\sqrt x} \frac x{d^2} \bigg) \\
&= \li(x) \bigg( \sum_{d=1}^\infty \frac{\mu(d)}{d\phi(d)} + O\bigg( \sum_{d>\log^2x} \frac1{d\phi(d)} \bigg) \bigg) + O\bigg( \frac x{\log^2x} \bigg) \\
&= \li(x) \prod_p \bigg( 1 + \frac{\mu(p)}{p\phi(p)} + 0 + \cdots\bigg) + O\bigg( \frac x{\log^2x} \bigg),
\end{align*}
which implies the statement of the lemma.
\end{proof}

The proof of Theorem~\ref{mnc theorem} is straightforward if we use the Wirsing--Odoni method; below is a statement of this method~\cite[Proposition~4]{FinchMartinSebah} with one hypothesis simplified for our purposes.

\begin{proposition} \label{WOmethod}
Let~$f$ be a multiplicative function. Suppose that $0\le f(p^r)\le 1$ for all primes~$p$ and all positive integers~$r$. Suppose also that there exist real numbers $\omega>0$ and $0<\beta<1$ such that 
\[
\sum_{p\le P}f(p)=\omega\frac P{\log{P}}+O\bigg( \frac P{(\log{P})^{1+\beta}}\bigg)
\]
as $P\to\infty$. Then the product over all primes 
\[
C_f=\frac 1{\Gamma(\omega)}\prod_p\bigg(1+\frac{f(p)}p+\frac{f(p^2)}{p^2}+\frac{f(p^3)}{p^3}+\cdots\bigg) \bigg(1-\frac 1p\bigg)^\omega
\]
converges (hence is positive), and 
\[
\sum_{n\le N}f(n)=C_fN(\log{N})^{\omega-1}+O_f(N(\log{N})^{\omega-1-\beta})
\]
as $N\to\infty$.
\end{proposition}

\begin{proof}[Proof of Theorem~\ref{mnc theorem}]
Fix $0<\ep<1$.
Let~$f$ be the indicator function of the set of integers~$n$ with the property that $\Z_n^\times$ is maximally non-cyclic. By Proposition~\ref{max non cyc characterization prop}, the function~$f$ is multiplicative with values on prime powers
\[
f(p^r) = \begin{cases}
1, &\text{if $p=2$ and $1\le r\le3$}, \\
1, &\text{if $p\ge3$ and $1\le r\le2$ and $p-1$ is squarefree}, \\
0, &\text{otherwise}.
\end{cases}
\]
By Lemma~\ref{sfsum}, this function satisfies the hypotheses of Proposition~\ref{WOmethod} with $\omega=\xi$ and $\beta=1-\ep$. We conclude from that proposition that
\[
\#\big\{ n\le x\colon \Z_n^\times \text{ is maximally non-cyclic} \} = \frac{Ax}{(\log x)^{1-\xi}} + O_\ep\bigg( \frac x{(\log x)^{2-\xi-\ep}} \bigg),
\]
where~$A$ is given by the convergent product
\begin{align*}
A&=\frac 1{\Gamma(\xi)}\prod_p\bigg(1+\frac{f(p)}{p}+\frac{f(p^2)}{p^2}+\frac{f(p^3)}{p^3}+\cdots\bigg) \bigg(1-\frac{1}{p}\bigg)^{\xi}\\
&=\frac 1{\Gamma(\xi)} {\bigg(1+\frac 12+\frac 14+\frac 18\bigg)} \lim_{x\to\infty}\bigg(\prod_{\substack{3\le p\le x\\p-1\text{ squarefree}}}\bigg(1+\frac 1p+\frac 1{p^2}\bigg) \prod_{2\le p\le x}\bigg(1-\frac 1p\bigg)^\xi\bigg)\\
&=\frac {15}{14\Gamma(\xi)}\lim_{x\to\infty}\bigg(\prod_{\substack{2\le p\le x\\p-1\text{ squarefree}}}\bigg(1+\frac 1p+\frac 1{p^2}\bigg) \prod_{2\le p\le x}\bigg(1-\frac 1p\bigg)^\xi\bigg)
\end{align*}
as claimed.
\end{proof}

\bibliography{mybibliography}
\bibliographystyle{plain}

\end{document}